\documentclass[12pt]{article}

\usepackage{t1enc,amsmath,amssymb,amsthm,amscd,epsfig,latexsym,amsfonts,ifthen}

\newcommand{\bt}{\begin{theo}}
\newcommand{\et}{\end{theo}}
\newcommand{\bl}{\begin{lem}}
\newcommand{\el}{\end{lem}}
\newcommand{\bc}{\begin{cor}}
\newcommand{\ec}{\end{cor}}
\newcommand{\bex}{\begin{ex}}
\newcommand{\eex}{\end{ex}}
\newcommand{\bp}{\begin{prop}}
\newcommand{\ep}{\end{prop}}
\newcommand{\bob}{\begin{rk}}
\newcommand{\eob}{\end{rk}}

\newcommand{\bdefn}{\begin{de}}
\newcommand{\edefn}{\end{de}}
\newcommand{\be}{\begin{equation}}
\newcommand{\ee}{\end{equation}}
\newcommand{\bea}{\begin{eqnarray}}
\newcommand{\eea}{\end{eqnarray}}
\newcommand{\bqa}{\begin{eqnarray*}}
\newcommand{\eqa}{\end{eqnarray*}}
\newcommand{\ba}{\begin{array}}
\newcommand{\ea}{\end{array}}
\newcommand{\bq}{\begin{quote}}
\newcommand{\eq}{\end{quote}}
\newcommand{\bi}{\begin{itemize}}
\newcommand{\ei}{\end{itemize}}
\newcommand{\ben}{\begin{enumerate}}
\newcommand{\een}{\end{enumerate}}
\newcommand{\bfl}{\begin{flushleft}}
\newcommand{\efl}{\end{flushleft}}
\newcommand{\bfr}{\begin{flushright}}
\newcommand{\efr}{\end{flushright}}
\newcommand{\bce}{\begin{center}}
\newcommand{\ece}{\end{center}}

\newtheorem{teo}{\sc Theorem}[section]

\newtheorem{lem}{\sc Lemma}[section]

\newcommand{\CC}{\mathbb{C}}

\newcommand{\ZZ}{\mathbb{Z}}

\newcommand{\modua}[2]{1-|{#1}|^{#2}}
\newcommand{\difztzg}{\zt-\zeta}
\newcommand{\difzgz}{\zeta-z}
\newcommand{\difztz}{\zt-z}
\newcommand{\difzgzt}{\zeta-\zt}

\newcommand{\difmodu}[4]{|{#1}|^{#2}-|{#3}|^{#4}}

\newcommand{\sumzgz}{\zeta+z}

\newcommand{\DD}{\mathbb{D}}
\newcommand{\dD}{\partial\DD}

\newcommand{\zbl}{\overline{{z}}} %zeta barra
\newcommand{\zbg}{\overline {\rm\zeta}} %zeta griego barra
\newcommand{\zt}{\tilde{\zeta}} % zeta griego tilde
\newcommand{\zbt}{\overline{\tilde{\zeta}}} % zeta griego barra tilde
\newcommand{\xzb}{\tilde{\xi}} % xi tilde
\newcommand{\gam}[1]{\gamma_{#1}} %funcion gamma
\newcommand{\chupa}[1]{\displaystyle \overline {#1}}
\newcommand{\etat}{\overline{\tilde{\eta}}}

\newcommand{\modu}[1]{|{#1}|}

\newcommand{\anal}[1]{\partial^{#1}_{\overline{\rm z}}} %analitica
\newcommand{\opneu}{\partial_{\mu}} %operador de neumann
\newcommand{\ibzg}[1]{\frac{1}{2\pi i}\int_{\mid \zeta \mid=1}{#1}\; \frac{d\zeta}{\zeta}} %frontera
\newcommand{\iintzgDOS}[1]{\frac{1}{2\pi}\int_{\mid\zeta\mid < 1}{#1}\; d\xi d\eta} %interior
\newcommand{\iintzgNOR}[1]{\frac{1}{\pi}\int_{\mid\zeta\mid < 1}{#1}\; d\xi d\eta} %interior

\newcommand{\ibzgwz}[1]{\frac{1}{2\pi i}\int_{\mid \zeta \mid=1}{#1}\; d\zeta} %frontera
\newcommand{\ibzgcuatro}[1]{\frac{1}{4\pi i}\int_{\mid \zeta \mid=1}{#1}\; \frac{d\zeta}{\zeta}} %frontera

\newcommand{\ibzgwzb}[1]{\frac{1}{2\pi i}\int_{\mid \zeta \mid=1}{#1}\; d\zbg}

\newcommand{\ibzt}[1]{\frac{1}{2\pi i}\int_{\mid \zt \mid=1}{#1}\;\frac{d\zt}{\zt}} %frontera
\newcommand{\iintztNOR}[1]{\frac{1}{\pi}\int_{\mid\zt\mid < 1}{#1}\; d\xzb d\tilde{\eta}} %interior
\newcommand{\ibztwz}[1]{\frac{1}{2\pi i}\int_{\mid \zt \mid=1}{#1}\;d\zt} %frontera
\newcommand{\logzzbg}{\log(1-z\zbg)}

\newcommand{\usualker}{\frac{1-\modu{\zeta} ^{2}}{1-\zbl\zeta}}

\begin{document}

\title{\vspace{-1in}\parbox{\linewidth}{\footnotesize\noindent}
 \vspace{\bigskipamount} \\
Combined boundary value problems for the nonhomogeneous tri-analytic equation
\thanks{ {\em 2010 AMS Subject Classifications: 35G15, 35C15, 35N25} 
\hfil\break\indent
{\em Keywords: Boundary value problems, tri-analytic equations, solvability conditions. 
}
\hfil\break\indent
{\em $\dagger$: aditeodoro@usb.ve ~~ $\ddagger$: cvanegas@usb.ve}}}

\author{ Antonio N. Di Teodoro$\dagger$ and Carmen J. Vanegas$\ddagger$ \\
Departamento de Matem\'aticas Puras y Aplicadas \\
Universidad Sim\'{o}n Bol\'{\i}var, Caracas 1080-A, Venezuela}

\date{}
\maketitle

\begin{abstract}
In the present article we present a particular combination of
boundary problems for the inhomogeneous tri-analytic
equation: the Neumann-(Dirichlet-Neuman) problem and the
(Dirichlet-Neumann)-Dirichlet problem.
In order to obtain the solution and solvability conditions we use
an iteration's process involving those corresponding to equations of lower order.
\end{abstract}

\section{Introduction}

The basic boundary value problems in complex analysis, the Schwarz, 
the Dirichlet and the Neumann problems have been
studied for higher order complex partial differential equations.
All kind of combinations of their could be posed which yields a
large variety of different problems. However, not all of these
problems are well-posed problems. Therefore we have to look for
solvability conditions.

Integral representations for solutions to higher order partial
differential equations can be obtained by an iteration's process from
the representation integral formulas for those corresponding to the first order
equations. This method has been used many times, instead see \cite{h2,h3,h4,h5} and references therein.
In this paper we apply also this procedure. Although this method
can be used in regular domains, we will restrict to the unit disc
in order to obtain explicit solvability conditions and solutions
for the problems treated here.

In this article we will study some new boundary value problems by
combining different boundary conditions.
We limited our study to the inhomogeneous tri-analytic equation
and let to a future work the generalization of these combined problems
for the inhomogeneous poly-analytic equation.

Next, we present basic problems which have been proved in \cite{h3}
\begin{teo}\label{teo12}

The Dirichlet problem for the inhomogeneous Cauchy-Riemann
equation in the unit disc
\begin{equation*}
\anal {} \omega= f\,in\,\DD, \quad \omega=\gamma\, on\,\DD
\end{equation*}
for $f \in L_{1}(\DD, \CC), \; \gamma \in C(\dD,\CC)$
is solvable if and only if for  $\modu{z} <1$ \\
\begin{equation}\label{eq121}
\ibzgwz {\frac{\zbl\gamma(\zeta)}{1-\zbl\zeta}}-\iintzgNOR
{\zbl\frac{f(\zeta)}{1-\zbl\zeta}}=0.
\end{equation}
The solution then is uniquely given by
\begin{equation}\label{eq122}
\omega(z)=\ibzgwz {\frac{\gamma(\zeta)}{\zeta-z}}- \iintzgNOR
{\frac{f(\zeta)}{\zeta-z}}.
\end{equation}
\end{teo}

\begin{teo}\label{teo13}
The Neumann problem for the inhomogeneous Cauchy-Riemann equation
in the unit disc
\begin{equation*}
\anal {} \omega =f\, in\, \DD ,\quad \opneu {\omega} =\gamma\,
on\,\dD,\quad \omega(0)= c
\end{equation*}
for $f \in C^{\alpha}(\chupa{\DD}, \CC),\;0<\alpha<1 \; \gamma\in
C(\dD,\CC),\; c\in \CC$
is solvable if and only if for $\modu {z} <1$ 
\begin{equation}\label{eq131}
\ibzg {\frac{\gamma(\zeta)}{1-\zbl\zeta}}+\ibzgwzb
{\frac{f(\zeta)}{1-\zbl\zeta}}+\zbl\iintzgNOR
{\frac{f(\zeta)}{(1-\zbl\zeta)^{2}}}=0.
\end{equation}
The solution then is uniquely given by
\begin{eqnarray}\label{eq132}
\omega(z)=c-\ibzg {\gamma (\zeta)\logzzbg}-\ibzg {\zbg f(\zeta) \logzzbg} \nonumber \\
-z\iintzgNOR {\frac{f(\zeta)}{\zeta(\zeta-z)}}.
\end{eqnarray}
\end{teo}

\vspace{0.5cm}
\begin{teo}\label{teo14}
The Dirichlet-Neumann problem for the in\-ho\-mogeneous Bit\-sadze
equation in the unit disc
\begin{equation*}
\anal {2} \omega =f\, on\,\DD,\quad \omega = \gam {0}\, in\, \dD,
\quad \opneu {\anal {} \omega} =\gam {1}\, on\, \dD,\quad \anal {}
{\omega} (0)=c
\end{equation*}
is uniquely solvable for $f \in L_{1}(\DD, \CC)\cap C(\dD, \CC),
\; \gam {0},\,\gam {1} \in C(\dD,\CC),\; c\in \CC$ if and only if
\begin{equation}\label{eq141}
c -\ibzgwz {\frac{\gam {0}(\zeta)}{1-\zbl\zeta}} + \iintzgNOR
{\usualker\frac{f(\zeta)}{\zeta}}=0.
\end{equation}
and
\begin{equation}\label{eq142}
\ibzg{(\gam {1}(\zeta)-\zbg
f(\zeta))\frac{1}{(1-\zbl\zeta)}}+\iintzgNOR {\frac{\zbl
f(\zeta)}{(1-\zbl\zeta)^{2}}}=0.
\end{equation}
The solution then is given by
\begin{multline}\label{eq143}
\omega(z)=c\zbl+\ibzgwz {\frac {\gam {0} (\zeta)}{\zeta -z}} +
\ibzg { [\gam {1}(\zeta) -\zbg f(\zeta)]\logzzbg \frac{1-\mid z
\mid^{2}}{z}} \\
\\
+ \iintzgNOR {\frac{\modu{\zeta}^{2}-\modu{z}^{2}}{(\zeta-z)} \frac{f(\zeta)}{\zeta}}.
\end{multline}
\end{teo}

The Dirichlet problem for the inhomogeneous
poly-analytic equation is proved in \cite{h3}. For the same
equation the Neumann and Dirichlet-Neumann problem are proved in
\cite{h5} and \cite{h4} respectively. In turn the Dirichlet-Neumann problem for the inhomogeneous
poly-analytic equation is solved in \cite{h4}. 
However we observe that the cases which are
considered in this paper can not be obtained direct for the
formula found there. 
  
In order to establish the new combined problems we need
some identities which we will prove using classical
results of complex analysis as Gauss'theorem, Cauchy's theorem and
Cauchy-Pompeiu Formula \cite{h6}.

\section{The Neumann-(Dirichlet-Neumann) problem}

In order to solve the boundary problem we need the following lemma.
\begin{lem}\label{lema31}
For $\displaystyle \modu{z}<1$ and $\displaystyle \modu{\zt}<1$ we
have:
\begin{description}
\item[$i.$] $\displaystyle\iintzgNOR {\frac{\zbg\zbl}{(1-\zbl\zeta)^{2}}} =\zbl^{2}.$
\item[$ii.$] $\displaystyle\iintzgNOR {\frac{\zbl}{(\zt-\zeta)(1-\zbl\zeta)^{2}}}=\frac{\zbl\zbt-2\zbl^{2}+\zbl^{3}\zt}{(1-\zbl\zt)^{2}}.$
\item[$iii.$] $\displaystyle\iintzgNOR {\frac{1-\modu{\zeta}^{2}}{\zeta}\log(1-\zeta\zbt)\frac{\zbl}{(1-\zbl\zeta)^2}}=-\frac{\zbl\zbt}{2}.$
\item[$iv.$] $\displaystyle\iintzgNOR {\frac{\modu{\zt}^{2}-\modu{\zeta}^{2}}{(\zt-\zeta)}\frac{\zbl}{(1-\zbl\zeta)^{2}}}
=\frac{\zbl\modu{\zt}^{2}(2\zt-4\zbl+2\zt\zbl^{2}-\zbt)+2\zbl^{2}}{2(1-\zbl\zt)^{2}}.$
\item[$v.$]  $\displaystyle\iintzgNOR {\frac{\zbg z}{\zeta(\zeta -z)}} =-\frac{\zbl^{2}}{2}.$
\item[$vi.$] $\displaystyle\iintzgNOR {\frac{z}{(\zt-\zeta)\zeta(\zeta-z)}}=\frac{\zbt-\zbl}{\zt-z}-\frac{\zbt}{\zt}.$
\item[$vii.$] $\displaystyle 
\iintzgDOS{\frac{\sumzgz}{(\difzgz)\zeta}\frac{\modua{\zeta}{2}}{\zeta}\log(1-\zeta\zbt)} \\
~~~~~~~~~~~~~~~~~~~~~~~~~~~~~~~~~~~~~~~~~~~ =\frac{\modu{z}^{4}-2\modu{z}^{2}+1}{2z^{2}}\log(1-z\zbt)+ 
\frac{\zbt}{2z}+ \frac{\zbt^{2}}{8}. $
\item[$viii.$] $\displaystyle\iintzgNOR {\frac{(\modu{\zt}^{2}-\modu{\zeta}^{2})z}{(\zt-\zeta)\zeta(\zeta -z)}}=\modu{\zt}^{2}\left[ \frac{\zbt-\zbl}{\zt-\zeta}-\frac{\zbt}{\zt}\right]-\frac{z(\zbt^{2}-\zbl^{2})}{2(\zt-z)}.$
\item[$ix.$] $\displaystyle 
\iintzgNOR{\frac{1-\modu{\zeta}^{2}}{\zeta}\log(1-\zeta\zbt)\frac{z}{\zeta(\zeta-z)}} \\
~~~~~~~~~~~~~~~~~~~~~~~~~~~~~~~~~~~~~~~~~~~ = \frac{\zbt}{2}+\frac{3\zbt^{2}}{4}+ \frac{(1-2\modu{z}^{2}+ \modu{z}^{4})\log(1-z\zbt)}{2z^{2}}.$
\end{description}
\end{lem}
\begin{proof}
\begin{description}
\item[$i.$] Using the Green-Gauss theorem we have
\begin{multline*}
\zbl\iintzgNOR {\frac{\zbg}{(1-\zbl\zeta)^{2}}}=\zbl\ibzgwz{\frac{\zbg^{2}}{2(1-\zbl\zeta)^{2}}}
=\zbl\ibzgwzb{\frac{-\zbg^{2}}{2(\zbg-\zbl)^{2}}} \\
\\
=\frac{\zbl}{2}\left[\overline{\ibzgwz{\frac{\zeta^{2}}{(\zeta-z)^{2}}}}\right]=
\frac{\zbl}{2}\overline{(\zeta^{2})'_{\zeta}\vert_{\zeta=z}}=\zbl^{2}.
\end{multline*}
\item[$ii.$] Using the Cauchy-Pompeiu formula we obtain
\begin{multline*}
\frac{\zbt}{(1-\zbl\zt)^{2}}=\ibzgwz{\frac{\zbg}{(1-\zbl\zeta)^{2}(\zeta-\zt)}}-\iintzgNOR{\frac{1}{(1-\zbl\zeta)^{2}(\zeta-\zt)}} \\
\\
=\overline{\ibzgwz{\frac{\zeta^{2}}{(\zeta-z)^{2}(1-\zbt\zeta)}}}-\iintzgNOR{\frac{1}{(1-\zbl\zeta)^{2}(\zeta-\zt)}} \\
\\
= \overline{\left(\frac{\zeta^{2}}{1-\zbt\zeta}\right)'_{\zeta}\vert_{\zeta=z}}-\iintzgNOR{\frac{1}{(1-\zbl\zeta)^{2}(\zeta-\zt)}}
\end{multline*}
which implies ~
$
\displaystyle\zbl\iintzgNOR{\frac{1}{(1-\zbl\zeta)^{2}(\zt-\zeta)}}=\frac{\zbl\zbt}{(1-\zbl\zt)^{2}}-\frac{\zbl(2\zbl-\zbl^{2}\zt)}{(1-\zbl\zt)^{2}}.
$
\item[$iii.$]
As ~ $\displaystyle\frac{\zbl(1-\modu{\zeta}^{2})}{\zeta(1-\zbl\zeta)^{2}}=\frac{\zbl}
{(1-\zbl\zeta)\zeta}+\frac{(\zbl-\zbg)}{(1-\zbl\zeta)}\frac{\zbl}{(1-\zbl\zeta)},$ ~
we have
\begin{multline*}
\iintzgNOR
{\frac{1-\modu{\zeta}^{2}}{\zeta}\log(1-\zeta\zbt)\frac{\zbl}{(1-\zbl\zeta)^2}}=\iintzgNOR
{\frac{\zbl\log(1-\zeta\zbt)}{\zeta(1-\zbl\zeta)}}\\
\\
+\iintzgNOR{\frac{(\zbl^{2}-\zbl\zbg)\log(1-\zeta\zbt)}{(1-\zbl\zeta)^{2}}}=\zbl\iintzgNOR
{\frac{\zbl\log(1-\zeta\zbt)}{(1-\zbl\zeta)}}\\
\\
+\zbl\iintzgNOR{\frac{\log(1-\zeta\zbt)}{\zeta}}+\iintzgNOR{\frac{(\zbl^{2}-\zbl\zbg)\log(1-\zeta\zbt)}{(1-\zbl\zeta)^{2}}}.
\end{multline*}
Due the Green-Gauss theorem the second integral of the last equality equals 
$$
\frac{\zbl}{\pi}\int_{\mid\zeta\mid < 1}
{\frac{\log(1-\zeta\zbt)~ d\xi d\eta}{\zeta}} =\frac{\zbl}{2\pi i}\int_{\mid\zeta\mid = 1}
{\frac{\zbg\log(1-\zeta\zbt)~ d\zeta}{\zeta}} 
=\zbl(\log(1-\zeta\zbt)){\zeta}\vert_{\zeta=0}=-\zbl\zbt
$$
and the third one equals
\begin{multline*}
\displaystyle\iintzgNOR
{\frac{(\zbl^{2}-\zbl\zbg)\log(1-\zeta\zbt)}{(1-\zbl\zeta)^{2}}}= \ibzgwz
{\frac{\log(1-\zeta\zbt)}{(1-\zbl\zeta)^{2}}(\zbl^{2}\zbg-\frac{\zbl\zbg^{2}}{2})} \\
\\
=\displaystyle\frac{\zbl^{2}\log(1-\zeta\zbt)}{(1-\zbl\zeta)^{2}}\Bigr\vert_{\zeta=0}-\zbl\left(
\frac{\log(1-\zeta\zbt)}{2(1-\zbl\zeta)^{2}}\right)'_{\zeta}\Biggr\vert_{\zeta=0}=\frac{\zbl\zbt}{2}.
\end{multline*}
Therefore  the identity is satisfied because the first integral is equal to zero.
\item[$iv.$] From the Cauchy integral formula we have
\begin{equation}\label{lemauno41}
\ibzgwz{\frac{\zbg}{\zeta^{3}(\zbg-\zbl)^{2}(1-\zt\zbg)}}=\frac{2\zbl-\zt\zbl^{2}}{(1-\zbl\zt)^{2}}
\end{equation}
and
\begin{equation}\label{lemauno42}
\ibzgwz{\frac{\zeta\zbg^{2}}{2(\zeta-\zt)(1-\zbl\zeta)^{2}}}=\frac{2\zbl}{2(1-\zbl\zt)^{2}}.
\end{equation}
Applying the Cauchy-Pompeiu formula and using (\ref{lemauno41}) and (\ref{lemauno42}) we obtain
\begin{multline*}
\displaystyle\frac{\modu{\zt}^{2}}{\pi}\int_{\mid\zeta\mid < 1}
{\frac{1}{(1-\zbl\zeta)^{2}(\zt-\zeta)}}=\frac{\modu{\zt}^{2}\zt}{(1-\zbl\zt)^{2}}-
\frac{\modu{\zt}^{2}}{2\pi i}\int_{\mid \zeta \mid=1}
{\frac{\zbg}{(1-\zbl\zeta)^{2}(\zeta-\zt)}} \\
\\
=\displaystyle\frac{\modu{\zt}^{2}(\zt-2\zbl+\zt\zbl^{2})}{(1-\zt\zbl)^{2}}
\end{multline*}
and
\begin{multline*}
\displaystyle\iintzgNOR{\frac{\zeta\zbg}{(1-\zbl\zeta)^{2}(\zt-\zeta)}}=\frac{\zt\zbt^{2}}{2(1-\zbl\zt)^{2}}
-\ibzgwz{\frac{\zeta\zbg^{2}}{2(1-\zbl\zeta)^{2}(\zeta-\zt)}} \\
\\
=\displaystyle\frac{\zbt\modu{\zt}^{2}-2\zbl}{2(1-\zt\zbl)^{2}}.
\end{multline*}
Adding the results of the two previous expressions we obtain the desired identity.
\item[$v.$] 
First we consider 
$$
\psi(z)=\ibzgwz{\frac{\zbg^{2}}{2(\zeta-z)}}~~ \mbox{and} ~~ 
\psi^{(k)}(z)=\frac{k!}{2\pi i}\int_{\mid \zeta \mid=1}
{\frac{\zbg^{2}}{2(\zeta-z)^{k+1}}},~ k\in\ZZ.
$$
We observe that $\psi$ is an holomorphic function respect to $z$.
Using the change $\displaystyle\zeta=\exp{i\theta},~ 0\leq\theta\leq 2\pi$, 
we can prove that $\displaystyle\psi^{(k)}(0)=0,\quad
k=0,1,\cdots$, which mean 
$\displaystyle \psi(z)=\sum^{\infty}_{k=0}\frac{\psi^{k}(0)z^{k}}{k!}=0.$
On the other hand, because of Cauchy-Pompeiu formula
\begin{equation*}
\frac{\zbl^{2}}{2}=\ibzgwz{\frac{\zbg^{2}}{2(\zeta-z)}}-\iintzgNOR{\frac{\zbg}{\zeta-z}}=-\iintzgNOR{\frac{\zbg}{\zeta-z}}
\end{equation*}
and if we make $\displaystyle z=0$ we get
\begin{equation*}
\iintzgNOR{\frac{\zbg}{\zeta}}=0.
\end{equation*}
So, we have
\begin{equation*}
\iintzgNOR{\frac{\zbg
z}{\zeta(\zeta-z)}}=\iintzgNOR{\frac{\zbg}{\zeta-z}}-\iintzgNOR{\frac{\zbg}{\zeta}}=-\frac{\zbl^{2}}{2}.
\end{equation*}
\item[$vi.$] After to use the Cauchy-Pompeiu formula
\begin{multline*}
\iintzgNOR{\frac{1}{(\zeta-\zt)(\zeta-z)}}=
\frac{1}{\zt-z}\left[\iintzgNOR{\frac{1}{\zt-\zeta}}+\iintzgNOR{\frac{1}{\zeta-z}}\right]\\
\\
=\frac{1}{\zt-z}\left[
\zbt-\ibzgwz{\frac{\zbg}{\zeta-\zt}}+\ibzgwz{\frac{\zbg}{\zeta-z}}-\zbl
\right]=\frac{\zbt-\zbl}{\zt-z},
\end{multline*}
and if $\displaystyle z=0$ we have $\displaystyle\iintzgNOR{\frac{1}{(\zt-\zeta)\zeta}}=\frac{\zbt}{\zt}.$\\
Then, we can write
\begin{multline*}
\displaystyle\iintzgNOR{\frac{z}{(\zt-\zeta)\zeta(\zeta-z)}}=\iintzgNOR{\frac{1}{(\zt-\zeta)(\zeta-z)}}
- \frac{1}{\pi}\int_{\mid\zeta\mid < 1} \frac{d\xi d\eta}{(\zt-\zeta)\zeta} \\
\\
=\displaystyle\frac{\zbt-\zbl}{\zt-z}-\frac{\zbt}{\zt}.
\end{multline*}
\item[$vii.$]Since 
$$\displaystyle \frac{\sumzgz}{\zeta(\difzgz)}\left( \frac{\modua{\zeta}{2}}{\zeta} \right)=\left( \frac{2}{\difzgz}-\frac{1}{\zeta} \right)
\left(\frac{1}{\zeta}-\zbg\right)=\frac{2}{\zeta(\difzgz)}-\frac{2\zbg}{\difzgz}-\frac{1}{\zeta^{2}}+\frac{\zbg}{\zeta}$$
we have
\begin{multline}\label{4int}
\iintzgDOS{\frac{\sumzgz}{(\difzgz)\zeta}\frac{\modua{\zeta}{2}}{\zeta}\log(1-\zeta\zbt)}
=\iintzgDOS{\frac{2\log(1-\zeta\zbt)}{\zeta(\difzgz)}} \\
\\
-\iintzgDOS{\frac{\log(1-\zeta\zbt)}{\zeta^{2}}}-\iintzgDOS{\frac{2\zbg\log(1-\zeta\zbt)}{\difzgz}} \\
\\
+\iintzgDOS{\frac{\zbg\log(1-\zeta\zbt)}{\zeta}}.
\end{multline}
Solving the four integrals in (\ref{4int}) we obtain the identity.
For the first one we have after de Cauchy-Pompeiu formula
$$
\iintzgNOR{\frac{\log(1-\zeta\zbt)}{\zeta(\difzgz)}}=\ibzgwz{\frac{\zbg\log(1-\zeta\zbt)}{\zeta(\difzgz)}}-\frac{\zbl\log(1-z\zbt)}{z}.
$$
Now we observe that the boundary integral is an holomorphic function respect to $z$ an denote it as $\Psi(z)$, so
$$
\Psi(z)=\ibzgwz{\frac{\zbg\log(1-\zeta\zbt)}{\zeta(\difzgz)}}=\Psi(0)+\sum_{k=1}^{\infty}\frac{\Psi^{k}(0)}{k!}z^{k}.
$$
Now we calculate $\Psi(0)$
$$
\Psi(0)=-\ibzgwz{\frac{\zbg}{\zeta^{2}}\sum_{n=1}^{\infty}\frac{(\zeta\zbt)^{n}}{n}}=-\ibzgwz{\sum_{n=1}^{\infty}\frac{\zeta^{n-3}\zbt^{n}}{n}}.
$$
Taking $\displaystyle \zeta=e^{i\theta}$ we have
$$
\Psi(0)=-\frac{1}{2\pi}\int_{0}^{2\pi}{\sum_{n=1}^{\infty}\frac{e^{i(n-2)\theta}\zbt^{n}}{n}}\;d\theta
=-\frac{\zbt^{2}}{2},
$$
where we have considered the uniform convergence of the series.
The derivatives of $\Psi(z)$ have the form
$$ \Psi^{(k)}(z)=\frac{k!}{2\pi
i}\int_{\modu{z}=1}\frac{\zbg\log(1-\zeta\zbt)}{\zeta(\zeta-z)^{k+1}}\;d\zeta $$
and
\begin{multline*}
\displaystyle\Psi^{(k)}(0)=\frac{k!}{2\pi i}\int_{\modu{z}=1}\frac{\zbg\log(1-\zeta\zbt)}{\zeta^{k+2}}\;d\zeta =
\frac{k!}{2\pi i}\int_{\modu{z}=1}\frac{\log(1-\zeta\zbt)}{\zeta^{k+3}}\;d\zeta \\
\\
= \displaystyle -\frac{k!}{2\pi i}\int_{\modu{z}=1}\sum_{n=1}^{\infty}\frac{(\zeta\zbt)^{n}}{n\zeta^{k+3}}\;d\zeta
= -\frac{k!}{2\pi i}\int_{\modu{z}=1}\sum_{n=1}^{\infty}\frac{\zeta^{n-(k+3)}\zbt^{n}}{n}\;d\zeta
\end{multline*}
and making the change $\zeta=e^{i\theta},\; \theta\in[0,2\pi]$
$$
\Psi^{(k)}(0)=-\frac{k!}{2\pi}\int_{0}^{2\pi}\sum_{n=1}^{\infty}e^{i(n-(k+2))\theta}\frac{\zbt^{n}}{n}\;d\theta=-\frac{k!\zbt^{k+2}}{k+2}.
$$
So we have 
$$
\Psi(z)=-\sum_{k=0}^{\infty}\frac{\zbt^{k+2}}{k+2}z^{k}=\frac{\log(1-z\zbt)}{z^{2}}+\frac{\zbt}{z}.
$$
Therefore we arrive to
\begin{multline*} 
\displaystyle \iintzgNOR{\frac{\log(1-\zeta\zbt)}{\zeta(\difzgz)}}=\frac{\log(1-z\zbt)}{z^{2}}+\frac{\zbt}{z}-\frac{\zbl\log(1-z\zbt)}{z} \\
\\
=\displaystyle\frac{(\modua{z}{2})\log(1-z\zbt)}{z^{2}}+\frac{\zbt}{z}.
\end{multline*}
Observing the former calculation we have for the second integral in (\ref{4int})
$$
\iintzgDOS{\frac{\log(1-\zeta\zbt)}{\zeta^{2}}}=\frac{\Psi(0)}{2}=-\frac{\zbt^{2}}{4}.
$$
For the third and fourth integral in (\ref{4int}) we have
\begin{multline*}
\displaystyle-\iintzgNOR{\frac{\zbg\log(1-\zeta\zbt)}{\difzgz}}=\frac{\zbl^{2}\log(1-z\zbt)}{2}
-\ibzgwz{\frac{\zbg^{2}\log(1-\zeta\zbt)}{2(\difzgz)}}\\
\\
= \displaystyle\frac{\zbl^{2}\log(1-z\zbt)}{2}-\frac{\Psi(z)}{2}
\end{multline*}
and 
\begin{eqnarray*}
\iintzgDOS{\frac{\zbg\log(1-\zeta\zbt)}{\zeta}}=\ibzgwz{\frac{\zbg^{2}\log(1-\zeta\zbt)}{4\zeta}}=\frac{\Psi(0)}{4}=
-\frac{\zbt^{2}}{8}
\end{eqnarray*}
respectively. 

\item[$viii.$] From
\begin{multline*}
\iintzgNOR{\frac{\zbg z}{(\zt-\zeta)(\zeta-z)}}=
\frac{z}{\zt-z}\left[\iintzgNOR{\frac{\zbg}{\zt-\zeta}}+\iintzgNOR{\frac{\zbg}{\zeta-z}}\right] \\
\\
= \frac{z}{\zt-z}\left[\frac{\zbt^{2}-\zbl^{2}}{2} \right]
\end{multline*}
and $(vi)$ of this lemma we get the result.
\item[$ix.$] It follows making
\begin{equation*}
\frac{z}{\zeta(\zeta-z)}\frac{1-\modu{\zeta}^{2}\log(1-\zeta\zbt)}{\zeta}=\left(
\frac{1}{\zeta-z}-\frac{1}{\zeta} \right)\left(
\frac{1}{\zeta}-\zbg \right)\log(1-\zeta\zbt)
\end{equation*}
and applying $(iii)$ of this lemma.
\end{description}
\end{proof}
\begin{teo}\label{teo33}
The Neumann-(Dirichlet-Neumann) problem for the inhomogeneous
tri-analytic equation in the unit disc
\begin{equation*}
\begin{array}{l}
\partial^{3}_{z}\,\omega=f\;in\,\DD,\quad \partial_{\nu}\,\omega =\gamma\;on\,\dD,\quad \omega(0)=c,\\\\
\partial_{\zbl}\,\omega=\gamma_{0}\;on\,\dD,\quad \partial_{\nu}\,\partial_{\zbl}\,\partial_{\zbl}\,\omega=\gamma_{1}\;on\,\dD,\quad \partial_{\zbl}\,\partial_{\zbl}\,\omega(0)=c_{1},
\end{array}
\end{equation*}
for $\displaystyle f \in C^{\alpha}(\overline{\DD}, \CC),\, 0<\alpha<1,\; \gamma,\gam {0},\gam {1} \in C(\dD,\CC),\; c,c_{1}\in \CC,$
is uniquely solvable if and only if for $\displaystyle z\in\DD$,
\begin{equation}\label{eq331}
c-\ibzgwz{\frac{\gamma_{0}(\zeta)}{1-\zbl\zeta}}+\iintzgNOR{f(\zeta)\frac{1-\modu{\zeta}^{2}}{\zeta(1-\zbl\zeta)}}=0,
\end{equation}
\begin{equation}\label{eq332}
\ibzgwz{\frac{(\gamma_{1}(\zeta)-\zbg
f(\zeta))}{\zeta(1-\zbl\zeta)}}+\iintzgNOR{\frac{\zbl
f(\zeta)}{(1-\zbl\zeta)^{2}}}=0,
\end{equation}
and
\begin{eqnarray}\label{eq333}
\ibzgwz{\frac{(\gamma(\zeta)+\zbl\gamma_{0}(\zeta))}{(1-\zbl\zeta)\zeta}}- \ibzgwz{\frac{(\gamma_{1}(\zeta)-\zbg
f(\zeta))\zbl\zbg}{2\zeta}}\nonumber\\
\nonumber\\
+\iintzgNOR{f(\zeta)\left[
\frac{2\zbl^{3}-2\zbl^{2}\zeta+\zbl\modu{\zeta}^{2}}{2(1-\zbl\zeta)^{2}}\right]}=0.
\end{eqnarray}
The Solution then is given by
\begin{multline}\label{eq334}
\omega(z)=c-\frac{c_{1}\zbl^{2}}{2}-\ibzg{\gamma(\zeta)\log(1-z\zbg)}
-\ibzgwz{\gamma_{0}(\zeta)\left[ \frac{\zbg-\zbl}{\zeta-z}-\frac{\zbg}{\zeta} \right]}\\
\\
-\ibzg{(\gamma_{1}(\zeta)-\zbg f(\zeta))\left[ \frac{\zbg}{2z}+\frac{3\zbg^{2}}{4}+\frac{\zbl^{2}}{2}-\frac{\zbl}{z}+\frac{1}{2z^{2}} \right]\log(1-z\zbg)}\\
\\
-\iintzgNOR{f(\zeta)\left[\frac{\zbg(\zbg-\zbl)}{\zeta-z}-\frac{z(\zbg^{2}+\zbl^{2})}{2\zeta(\zeta-z)}-\frac{\zbg^{2}}{\zeta}\right]}
\end{multline}
\end{teo}
\begin{proof}
The given system is converted into the following two boundary problems:
\begin{equation*}
\begin{array}{ll}
\partial_{\zbl}\,\omega=\varphi ~~ \mbox{in}~~ \DD,\quad \partial_{\nu}\,\omega =\gamma ~~ \mbox{on} ~~ \dD,\quad \omega(0)=c,\\\\
\partial_{\zbl}\,\partial_{\zbl}\,\varphi=f\; ~~ \mbox{in} ~~ \DD,\quad \varphi=\gamma_{0} ~~ \mbox{on} ~~ \dD,\quad \partial_{\nu}\,(\partial_{\zbl}\,\varphi)=\gamma_{1} ~~ \mbox{on} ~~ \dD,\quad \partial_{\zbl}\,\omega(0)=c_{1}.
\end{array}
\end{equation*}
So using Theorem \ref{teo13}, $\omega$ is
\begin{equation}\label{eq335}
\omega(z)=c-\ibzg{(\gamma(\zeta)-\zbg\varphi(\zeta))\log(1-z\zbg)}
-\iintzgNOR{\frac{z\varphi(\zeta)}{\zeta(\zeta-z)}}
\end{equation}
if and only if (\ref{eq131}) is satisfied with $\varphi$ instead of $f$,
and by Theorem \ref{teo14} $\varphi$ is
\begin{multline}\label{eq336}
\varphi(\zeta)=c_{1}\zbl+\ibzgwz{\frac{\gamma_{0}(\zeta)}{\zeta-z}}
+\ibzg{(\gamma_{1}(\zeta)-\zbg f(\zeta))\frac{1-\modu{z}^{2}}{z}\log(1-z\zbg)} \\
\\
+\iintzgNOR{f(\zeta)\frac{\modu{\zeta}^{2}-\modu{z}^{2}}{\zeta(\zeta-z)}}
\end{multline}
under the solvability condition (\ref{eq141}) and (\ref{eq142}) with $c_{1}$ instead $c.$ 
Now we consider (\ref{eq131}) with $\varphi$ instead of $f$:
\begin{equation}\label{eq337}
\ibzgwz{\frac{\gamma(\zeta)-\zbg\varphi(\zeta)}{(1-\zbl\zeta)\zeta}}+\iintzgNOR{\frac{\zbl\varphi(\zeta)}{(1-\zbl\zeta)^{2}}}=0.
\end{equation}
Substituting the expression for $\varphi$ into (\ref{eq337}) we have
\begin{multline*}
\ibzg{\frac{\gamma(\zeta)}{(1-\zbl\zeta)}}-c_{1}\ibzgwz{\frac{\zbg^{2}}{\zeta(1-\zbl\zeta)}}\\
\\
-\ibztwz{\gamma_{0}(\zt)\left[  \ibzg{\frac{\zbg}{(\zt-\zeta)(1-\zbl\zeta)}} \right]}\\
\\
-\iintztNOR{\frac{f(\zt)}{\zt}\left[ \ibzgwz{\frac{\modu{\zt}^{2}-\modu{\zeta}^{2}}{\zt-\zeta}\frac{\zbg}{\zeta(1-\zbl\zeta)}}  \right]}\\
\\
+c_{1}\iintzgNOR{\frac{\zbg\zbl}{(1-\zbl\zeta)^2}}
+\ibztwz{\gamma_{0}(\zt)\left[\iintzgNOR{\frac{\zbl}{(\zt-\zeta)(1-\zbl\zeta)^2}}\right]}
\end{multline*}
\begin{multline*}
+\ibzt{(\gamma_{1}(\zt)-\zbt
f(\zt))\left[\iintzgNOR{\frac{(1-\modu{\zeta}^{2})\zbl\log(1-\zeta\zbt)}{\zeta(1-\zbl\zeta)^{2}}}\right]}\\
\\
+\iintztNOR{\frac{f(\zt)}{\zt}\left[\iintzgNOR{\frac{\zbl(\modu{\zt}-\modu{\zeta}^{2})}{(\zt-\zeta)(1-\zbl\zeta)^{2}}}\right]}=0,
\end{multline*}
for $\displaystyle{\zt = \tilde\xi +i \etat}.$
In order to obtain (\ref{eq333}) we use $(i) - (iv)$ of Lemma \ref{lema31} and
$$
\ibzg{\frac{\zbt^{2}}{1-\zbl\zeta}}=\zbl^{2}, ~ 
\ibzg{\frac{\zbt}{(\zt-\zeta)(1-\zbl\zeta)}} = -\frac{\zbl^{2}}{1-\zbl\zt}
$$
and
$$
\ibzg{\frac{\modu{\zt}^{2}-\modu{\zeta}^{2}}{(\zt-\zeta)}\frac{\zbg}{(1-\zbl\zeta)}}=\frac{\zbl^{2}(1-\modu{\zt}^{2})}{1-z\zt}.
$$
which are calculated by applying of Cauchy integral formula.

In order to obtain (\ref{eq334}) we carry (\ref{eq336}) to (\ref{eq335}) having
\begin{multline*}
\omega(z)=c-\ibzg{\gamma(\zeta)\log(1-z\zbg)}-c_{1}\iintzgNOR{\frac{z\zbg}{\zeta(\zeta-z)}}\\
\\
-\ibztwz{\gamma_{0}(\zt)\left[\iintzgNOR{\frac{z}{(\zt-\zeta)\zeta(\zeta-z)}}\right]}\\
\\
- \frac{1}{2 \pi i} \int_{\mid \zt \mid=1}(\gamma_{1}(\zt)-\zbt f(\zt))
\left[\frac{1}{\pi}\int_{\mid \zeta \mid < 1}
\frac{(1-\modu{\zeta}^{2})z\log(1-\zeta\zbt)}{\zeta^{2}(\zeta-z)}d\xi d\eta \right]\frac{d\zt}{\zt} \\
\\
-\iintztNOR{\frac{f(\zt)}{\zt}\left[\iintzgNOR{\frac{(\modu{\zt}^{2}-\modu{\zeta}^{2})z}{(\zt-\zeta)\zeta(\zeta-z)}}\right]} 
+ c_{1}\ibzg{\zbg^{2}\log(1-z\zbg)}\\
\\
+ \ibztwz{\gamma_{0}(\zt) \left[\ibzg{\frac{\zbg\log(1-z\zbg)}{\zt-\zeta}}\right]}\\
\\
+\iintztNOR{\frac{f(\zt)}{\zt}\left[\ibzg{\frac{(\modu{\zt}^{2}-\modu{\zeta}^{2})\bar{\zeta} \log(1-z\bar{\zeta})}{\zt-\zeta}}\right]}.
\end{multline*}
Taking into account that the integrals
$
\displaystyle\ibzg{\zbg^{2}\log(1-z\zbg)}$, 
$$
\ibzg{\frac{\zbg\log(1-z\zbg)}{\zt-\zeta}}, ~~ \mbox{and} ~~
\ibzg{\frac{(\modu{\zt}^{2}-\modu{\zeta}^{2})\zbg\log(1-z\zbg)}{\zt-\zeta}}
$$
are all equal to zero because of the Cauchy integral formula and
using ${\it (v)-(ix)}$ of Lemma \ref{lema31} 
we get the solution (\ref{eq334}).
\end{proof}

\section{The (Dirichlet-Neumann)-Dirichlet problem}

Now we will study the combined problem (Dirichlet-Neumann)-Dirichlet.
As we did in the former problem, we will prove some identities.
\begin{lem}\label{lema34}
For $\displaystyle \modu{z}<1$ and $\displaystyle \modu{\zt}<1$ we have
\end{lem}
\begin{description}
\item[$i.$] $\displaystyle \iintzgNOR{\frac{\modua{\zeta}{2}}{\zeta(1-\bar{z}\zeta)(\zt-z)}}=\frac{1}{2\zt}\frac{2\zbt-\zt(\zbt^{2}+\zbl^{2})}
{(1-\bar{z}\zt)}.$
\item[$ii.$]$\displaystyle
\ibzg{\frac{\zbg}{\zt-\zeta}\frac{\modua{z}{2}}{z}\log(1-z\zbg)}=0.$
\item[$iii.$]$\displaystyle \iintzgNOR{\frac{\difmodu{\zeta}{2}{z}{2}}{\zeta(\difzgz)(\difztzg)}}
=\frac{1}{2\zt}\frac{\zbt(\modu{\zt}^{2}-2\modu{z}^{2})+\zbl^{2}(2-\zt)}{\difztz}.$
\end{description}
\begin{proof}
\begin{description}
\item[$i.$]Since ~ $\displaystyle \frac{1}{\zeta(1-\zbl\zeta)(\difztzg)}=\frac{1}{\zt(1-\zbl\zeta)}\left[\frac{1}{\zeta}+\frac{1}{\difztzg}\right]$
~ then
\begin{multline*}
\iintzgNOR{\frac{\modua{\zeta}{2}}{\zeta(1-\zbl\zeta)(\difztzg)}}=\iintzgNOR{\frac{1}{\zt\zeta(1-\zbl\zeta)}}\\
\\
+\iintzgNOR{\frac{1}{\zt(1-\zbl\zeta)(\difztzg)}}
-\iintzgNOR{\frac{\zbg}{(1-\zbl\zeta)(\difztzg)}}.
\end{multline*}
After applying Cauchy Pompeiu formula we have
\begin{multline*}
-\iintzgNOR{\frac{1}{(1-\zbl\zeta)(\difzgzt)}}=\frac{\zbt}{1-\zbl\zt}-\ibzgwz{\frac{\zbg}{(1-\zbl\zeta)(\difzgzt)}}\\
\\
=\frac{\zbt}{1-\zbl\zt}
-\overline{\ibzgwz{\frac{\zeta}{(\difzgz)(1-\zbt\zeta)}}}=\frac{\zbt}{1-\zbl\zt}-\frac{\zbl}{1-\zbl\zt}=\frac{\zbt-\zbl}{1-\zbl\zt}
\end{multline*}
and if we take $\displaystyle \zt=0$ we obtaing $\displaystyle
\iintzgNOR{\frac{1}{(1-\zbl\zeta)\zeta}}= \bar{z}.$ On the other hand
\begin{multline*}
\iintzgNOR{\frac{\zbg}{(1-\zbl\zeta)(\difzgzt)}}=
\ibzgwz{\frac{\zbg^{2}}{2(1-\zbl\zeta)(\difzgzt)}}-\frac{\zbt^{2}}{2(1-\zbl\zeta)}\\
\\
=\overline{\ibzgwz{\frac{\zeta^{2}}{2(\difzgz)(1-\zbt\zeta)}}}-\frac{\zbt^{2}}{2(1-\zbl\zeta)}=
\frac{\zbl^{2}-\zbt^{2}}{2(1-\zbl\zt)}.
\end{multline*}
\item[$ii.$] Follows from the Cauchy's theorem.
\item[$iii.$]
$$
\frac{1}{\pi}\int_{\mid\zeta\mid < 1}{\frac{(\difmodu{\zeta}{2}{z}{2})~~ d\xi d\eta}{\zeta(\difzgz)(\difztzg)}}=
\frac{1}{\pi}\int_{\mid\zeta\mid < 1}{\frac{\zbg ~~ d\xi d\eta}{(\difzgz)(\difztzg)}}
- \frac{\zbl}{\pi}\int_{\mid\zeta\mid < 1}{\frac{z ~~ d\xi d\eta}{\zeta(\difzgz)(\difztzg)}}.
$$
Since ~ $\displaystyle
\frac{\zbg}{(\difzgz)(\difztzg)}=\frac{\zbg}{\difztz}\left(\frac{1}{\difzgz}+\frac{1}{\difztzg}\right)$ ~ we have
\begin{eqnarray*}
\frac{1}{\pi}\int_{\mid\zeta\mid < 1}{\frac{\zbg ~ d\xi d\eta}{(\difzgz)(\difztzg)}}=
\frac{1}{\difztz}\frac{1}{\pi}\int_{\mid\zeta\mid < 1}{\frac{\zbg ~~ d\xi d\eta}{\difzgz}}
+\frac{1}{\difztz}\frac{1}{\pi}\int_{\mid\zeta\mid < 1}{\frac{\zbg ~~ d\xi d\eta}{\difztzg}}\\
\\
= \frac{1}{\difztz}\left(\frac{-\zbl^{2}}{2}+\frac{\zbt^{2}}{2}\right)
\end{eqnarray*}
where we used the proof of $(v)$ in Lemma \ref{lema31}. On the other side using 
$(vi)$ of Lemma \ref{lema31} we have ~
$
\displaystyle \zbl\iintzgNOR{\frac{z}{\zeta(\difzgz)(\difztzg)}}=\zbl\left(\frac{\zbt-\zbl}{\difztz}-\frac{\zbt}{\zt}\right).
$
\end{description}
\end{proof}
\begin{teo}
The (Dirichlet-Neumann)-Dirichlet problem for the inhomogeneous
tri-analytic equation in the unit disc
\begin{equation*}
\begin{array}{l}
\partial^{3}_{z}\,\omega=f ~~\mbox{in}~~ \DD,\quad \omega =
\gamma_{0}~~ \mbox{on} ~~\dD,\quad \partial_{\nu}\,\partial_{\zbl}\,\omega=\gamma_{1}~~ \mbox{on} ~~\dD,\quad\\
\partial^{2}_{\zbl}\,\omega=\gamma~~ \mbox{on} ~~\dD,\quad \partial_{\zbl}\,\omega(0)= c.
\end{array}
\end{equation*}
for $\displaystyle f \in L_{1}(\overline{\DD}, \CC)\cap
C(\overline{\dD}, \CC),\; \gamma,\gam {0},\gam {1} \in
C(\dD,\CC),\; c\in \CC,$ is uniquely solvable if and only if for
$\displaystyle z\in\DD$,
\begin{equation}\label{eq351}
\ibzgwz{\gamma(\zeta)\frac{\zbl}{1-\zbl\zeta}}=\iintzgNOR{f(\zeta)\frac{\zbl}{1-\zbl\zeta}},
\end{equation}
\begin{eqnarray}\label{eq352}
c-\ibzgwz{\gamma_{0}(\zeta)\frac{1}{1-\zbl\zeta}}+
\frac{1}{4\pi i}\int_{\modu{\zeta}=1}\frac{\gamma{(\zeta})}{\zeta}\frac{[2\zbg-\zeta(\zbg^{2}+\zbl^{2})]}{1-\zbl\zeta}\;d\xi\nonumber \\
\nonumber \\
-\frac{1}{2\pi}\int_{\modu{\zeta}<1}\frac{f{(\zeta})}{\zeta}\frac{[2\zbg-\zeta(\zbg^{2}+\zbl^{2})]}{1-\zbl\zeta}\;
d\xi d\eta=0
\end{eqnarray}
and
\begin{equation}\label{eq353}
\ibzgwz{\gamma_{1}(\zeta) +\zbl\gamma(\zeta)\frac{1}{\zeta(1-\zbl\zeta)}}-\iintzgNOR{f(\zeta)\frac{\zbl(\zbg-\zbl)}{(1-\zbl\zeta)^{2}}}=0.
\end{equation}
The solution then is given by
\begin{multline}\label{eq354}
\omega(z)=c\zbl+\ibzgwz{\frac{\gamma_{0}(\zeta)}{\difzgz}}+\ibzg{\gamma_{1}(\zeta)\frac{\modua{z}{2}}{z}\log(1-z\zbg)}\\
\\
+\ibzgcuatro{\gamma(\zeta)\left[\frac{\zbg(1-2\modu{z}^{2})+\zbl^{2}(2-\zeta)}{\difzgz}\right]}\\
\\
-\iintzgDOS{f(\zeta)\left[\frac{\zbg(\modu{\zeta}^{2}-2\modu{z}^{2})+\zbl^{2}(2-\zeta)}{\zeta(\difzgz)}\right]}.
\end{multline}
\end{teo}
\begin{proof}
The problem is discomposed into the system
\begin{equation}\label{eq355}
\partial_{\zbl}\partial_{\zbl}\,\omega=\varphi ~~ \mbox{in} ~~ \DD,\quad \omega =\gamma_{0} ~~ \mbox{on} ~~ \dD,\quad \partial_{\nu}\,\partial_{\zbl}\,\omega=\gamma_{1}, ~~ \mbox{on} ~~ \dD,\quad \partial_{\zbl}\,\omega(0)=c
\end{equation}
and
\begin{equation}\label{eq356}
\quad \partial_{\zbl}\,\varphi=f ~~ \mbox{in} ~~ \DD,\quad
\varphi=\gamma ~~ \mbox{on} ~~ \dD .
\end{equation}
By Theorem \ref{teo14}, the solution of (\ref{eq355}) is (\ref{eq143}) with $\varphi$ instead 
of $f$ under the soluability conditions (\ref{eq141}) and (\ref{eq142}) again with $\varphi$ instead of $f$.
On the other hand, the solution of (\ref{eq356}) is given by (\ref{eq122}) restricted to the condition (\ref{eq121}).
Substituting (\ref{eq122}) with $\varphi$ instead of $\omega$ in the
solution of (\ref{eq355}) we obtain
\begin{multline*}
\omega(z)=c\zbl+\ibzg{\frac{\gamma_{0}(\zeta)}{\difzgz}}+\ibzg{\gamma_{1}(\zeta)\frac{\modua{z}{2}}{z}\log(1-z\zbg)}\\
\\
-\ibztwz{\gamma(\zt)\left[\ibzg{\frac{\zbg}{\difztzg}\frac{\modua{z}{2}}{z}\log(1-z\zbg)} \right]}\\
\\
+\iintztNOR{f(\zt)\left[\ibzg{\frac{\zbg}{\difztzg}\frac{\modua{z}{2}}{z}\log(1-z\zbg)} \right]}\\
\\
+\ibztwz{\gamma(\zt)\left[ \iintzgNOR{\frac{\difmodu{\zeta}{2}{z}{2}}{\zeta(\difzgz)(\difztzg)}}\right]}\\
\\
-\iintztNOR{f(\zt)\left[\iintzgNOR{\frac{\difmodu{\zeta}{2}{z}{2}}{\zeta(\difzgz)(\difztzg)}}\right]}
\end{multline*}
After $\displaystyle\ibzg{\frac{\zbg\log(1-z\zbg)}{\zt-\zeta}}=0$
and $(iii)$ of Lemma \ref{lema34}, we have (\ref{eq354}).
In order to prove (\ref{eq352}) and (\ref{eq353}), we substitute
(\ref{eq122}) with $\varphi$ instead of $\omega$ in (\ref{eq141}) and
(\ref{eq142}) where we have taken $\varphi$ instead of $f$. It
yields,
\begin{multline*}
c-\ibzgwz{\frac{\gamma_{0}(\zeta)}{1-\zbl\zeta}}+\ibztwz{\gamma(\zt)
\left[\iintzgNOR{\frac{\modua{\zeta}{2}}{\zeta(1-\zbl\zeta)(\difztzg)}}\right]}\\
\\
-\iintztNOR{f(\zt)\left[\iintzgNOR{\frac{\modua{\zeta}{2}}{\zeta(1-\zbl\zeta)(\difztzg)}}
\right]}=0
\end{multline*}
and
\begin{multline*}
0 = \ibzgwz{\frac{\gamma_{1}(\zeta)}{\zeta(1-\zbl\zeta)}} \\
\\
-\frac{1}{2\pi i}\int_{|\zt|=1}\gamma(\zt)\left[\ibzg{\frac{\zbg}{(\difztzg)(1-\zbl\zeta)}}\right.
\left.-\iintzgNOR{\frac{\zbl}{(\difztzg)(1-\zbl\zeta)^{2}}}\right]\; d\zt 
\end{multline*}
\begin{multline*}
+ \frac{1}{\pi}\int_{|\zt|<1}f(\zt)\left[ \ibzg{\frac{\zbg}{(\difztzg)(1-\zbl\zeta)}}  \right.
\left. - \iintzgNOR{\frac{\zbl}{(\difztzg)(1-\zbl\zeta)^{2}}} \right]
\;d\tilde{\xi}d\tilde{\eta}.
\end{multline*}
Using $(i)$ of Lemma \ref{lema34} and $(ii)$, $(vi)$ of Lemma
\ref{lema31} we get the solvability conditions for this problem.
\end{proof}

\bob
The combined boundary value value problems studied in this paper can be generalized to 
the following combined problems for the nonhomogeneous poly-analytic equation:
$k$-Neumann-($m$-Dirichlet-$n$-Neumann),
($m$-Dirichlet-$n$-Neumann)-$k$-Dirichlet,
$k$-Dirichlet-($n$-Neumman-$m$-Dirichlet) and
($n$-Neumann-$m$-Dirichlet)-$k$-Neumann which extend the cases
treated in \cite{h4}.
\eob

\end{document}